\documentclass[a4paper,11pt, reqno]{amsart} 
\usepackage{amssymb,amsthm,amsmath}

\usepackage{amsfonts}
\usepackage{amscd}
\usepackage{amssymb}
\usepackage{enumerate}
\usepackage{graphicx}
\allowdisplaybreaks
\usepackage{mathabx}
\usepackage{color}
\usepackage{amsbsy}
\usepackage{graphicx}
\usepackage{amsthm}
\usepackage{amsmath}
\usepackage{amsxtra}
\usepackage{mathrsfs}
\usepackage{bbm}
\usepackage{dsfont}
\usepackage{ifthen}
\usepackage[colorlinks,citecolor=red,pagebackref,hypertexnames=false]{hyperref} 

\newtheorem{theorem}{Theorem}[section]

\newtheorem{lem}[theorem]{Lemma}

\newtheorem{prop}[theorem]{Proposition}

\theoremstyle{definition}

\theoremstyle{remark}
\newtheorem{rem}[theorem]{Remark}
\numberwithin{equation}{section}
\definecolor{red}{rgb}{1.0, 0.0, 0.0}
\newcommand{\Bea}{\begin{eqnarray*}}
	\newcommand{\Eea}{\end{eqnarray*}}
\newcommand{\Be} {\begin{equation*}}
	\newcommand{\Ee} {\end{equation*}}
\newcommand{\be} {\begin{equation}}
	\newcommand{\ee} {\end{equation}}
\newcommand{\bea} {\begin{eqnarray}}
	\newcommand{\eea} {\end{eqnarray}}

\renewcommand{\Re}{\operatorname{Re}}
\renewcommand{\Im}{\operatorname{Im}}

	\title[Restriction Theorem and Strichartz estimate for orthonormal functions  ]
	{Restriction Theorem and Strichartz estimate for orthonormal functions associated with the Special Hermite Operator}

	\author{Sunit Ghosh and Jitendriya Swain}

	\address{Sunit Ghosh  \endgraf Department of Mathematics
		\endgraf IIT Guwahati
		\endgraf Guwahati, Assam, India.}
	\email{g.sunit@iitg.ac.in, ghosh.sunit0197@gmail.com }

	\address{Jitendriya Swain,  Associate professor  \endgraf Department of Mathematics
			\endgraf IIT Guwahati
			\endgraf Guwahati, Assam, India.}
	\email{jitumath@iitg.ac.in}

	\date{\today}

\allowdisplaybreaks
 
\begin{document}
	
	\begin{abstract} Let $\mathcal{L}$ be the special Hermite operator on $\mathbb{C}^n$. As a continuation of the recent results in \cite{SG}, we establish new Strichartz estimates for systems of orthonormal functions associated with general flows of the form $e^{-it\phi(\mathcal{L})}$, where $ \phi : \mathbb{R}^{+} \to \mathbb{R} $ is a smooth function. Our approach relies on restriction estimates for the Fourier–special Hermite transform on the class of surfaces $\{(\lambda, \mu, \nu)\in \mathbb{R}\times\mathbb{N}_0^n\times\mathbb{N}_0^n : \lambda=\phi(2|\nu|+n)\}$.  We also discuss the endpoint case of the orthonormal Strichartz estimate for the Schr\"{o}dinger propagator $e^{-it\mathcal{L}}$. Furthermore, we generalize restriction estimates for the special Hermite spectral projections in the context of trace ideals (Schatten spaces).
\end{abstract}
\maketitle
\section{Introduction}
Consider the special Hermite operator (also called twisted Laplacian) $\mathcal{L}$ on $\mathbb{C}^n, n \geq 1$, defined by 
$$
\mathcal{L}= - \Delta_z + \frac{1}{4} |z|^2  - i \sum_{j=1}^{n}\left(x_j\frac{\partial}{\partial{y_j}} - y_j\frac{\partial}{\partial{x_j}}\right),
$$
 with $z=x+iy\in\mathbb{C}^n$. The spectrum of this operator is discrete and consists of points $2k + n, k \in \mathbb{N}_0 = \mathbb{N}\cup\{0\}$ and the eigenspaces associated to each of these eigenvalues are infinite dimensional. These eigenspaces are spanned by $\Phi_{\mu,\nu}, \mu, \nu \in \mathbb{N}^n_0$, $|\nu|=\nu_1+\cdots+\nu_n = k$, where $\Phi_{\mu, \nu}$ are the special Hermite functions (defined in Section \ref{sec2}).
For $\mu,\nu\in\mathbb{N}_0^n$, let $\hat{f}(\mu,\nu)=\langle f,\Phi_{\mu,\nu}\rangle$ denote the special Hermite coefficients. 

Let $ \phi : \mathbb{R}^{+} \to \mathbb{R} $ be a smooth function. The function $e^{-it\phi(\mathcal{L})}f$ denotes the solution to the initial value problem
\begin{align}\label{euq1}
	i \partial_t u(t, z)-{\phi(\mathcal{L})}u(t, z) &= 0,\quad z \in \mathbb{C}^n, \hspace{2pt} t \in \mathbb{R}\setminus \{0\}, \\
	\nonumber u(0,z) &= f(z).
\end{align}
The first part of the paper is concerned with extended versions of Strichartz estimates of the form
 \begin{align}\label{euq2}
			\left\|\sum_{j\in J} n_{j }\left|e^{-i t \phi(\mathcal{L})}f_{j }\right|^{2}\right\|_{L^p (I, L^q(\mathbb{C}^n))} \leqslant C N^\sigma\left(\sum_{j \in J}\left|n_{j}\right|^{\beta}\right)^{1/\beta},
		\end{align}
for all orthonormal system $(f_j)_j$ in $L^2(\mathbb{C}^n)$ with $supp \hat{f}_j \subset \{(\mu,\nu): \sqrt{2|\nu|+n }\leq N\}$ and all sequence $ (n_{j })_j \subset \mathbb{C}$, for a bounded interval $I\subset \mathbb{R}.$

More precisely, we prove (\ref{euq2}) under the assumption that $\phi$ is of power type  near $\infty,$ see (\ref{e1}). By choosing $\phi$ appropriately we obtain Strichartz estimates for the wave, Klein Gorden and fractional Schr\"{o}dinger equation associated with $\mathcal{L}$. Moreover, the estimate (\ref{euq2}) can be upgraded to a global-type estimate (see Theorem \ref{MAIN3}), provided the orthonormal system $(f_j)_j$ lies in the Sobolev space $H^s(\mathcal{L})$ with $s > \frac{\sigma}{2}$. To the best of our knowledge, such an estimate is not available in the literature in this setting.

Motivation to study estimates of the form (\ref{euq2}) arise in the context of many-body quantum mechanics. A system of $N$ independent fermions is described by a collection of $N$ orthonormal functions $f_1, \ldots, f_N$ in $L^2$. So functional inequalities that incorporate a significant number of orthonormal functions are highly valuable for the mathematical analysis of large-scale quantum systems (see, for instance, Lieb-Thirring \cite{lieb03, lieb05},  Frank-Lewin-Lieb-Seiringer \cite{frank}, Lewin-Sabin \cite{lewin1, lewin2}, Frank-Sabin\cite{frank1}, and the references therein). The idea in this line of investigation is to generalize the classical inequalities for a single-function to an orthonormal system. In \cite{frank1}, Frank and Sabin obtain  Fourier restriction estimate in Schatten spaces for the smooth compact surfaces in $\mathbb{R}^n$ with non-zero Gauss curvature (Stein-Tomas \cite{T1}) and by duality they generalize Stein-Tomas estimate for orthonormal  system (see also \cite{frank2}). In the same work, they also proved Fourier restriction estimate for quadratic surfaces  in $\mathbb{R}^n$ in the context of Schatten spaces. As an application they generalize the classical Strichartz estimates  \cite{st} to the orthonormal system for the Schr\"{o}dinger, wave, and Klein–Gordon equations (see \cite{lee, lee1} for more recent results). Strichartz estimates for orthonormal systems associated with the Schr\"{o}dinger propagator associated with several self-adjoint operators have been obtained in several works. For instance, we refer to Mondal–Swain \cite{ssm} for Hermite operator, Ghosh-Mondal-Swain \cite{SG} for special Hermite operator, Feng–Song \cite{Feng} for the Laguerre operator, Mondal–Song \cite{SSM} for $(k,a)$-generalized Laguerre operators, and Nakamura \cite{nakamura}, Wang-Zhang-Zhang \cite{WZZ} for Laplace Beltrami operator on compact manifolds.
 For more general results covering a wide class of dispersive equations see Hoshiya \cite{Hoshiya}, Feng-Mondal-Song-Wu \cite{Feng1}.

In \cite{SG}, the authors established a global Strichartz estimate for orthonormal systems associated with the Schr\"{o}dinger semigroup $e^{-it\mathcal{L}}$ by employing a restriction estimate for the Fourier–special Hermite transform on certain discrete surfaces. In this work, we derive a local restriction estimate for the Fourier–special Hermite transform on a broader class of discrete surfaces, which in turn yields an estimate of the form (\ref{euq2}).

For $F \in L^1(\mathbb{R}\times \mathbb{C}^n)$ the Fourier-special Hermite transform of $F$ is given by
\begin{align}\label{Fourier-special Hermite transform}
	\hat{F}(\lambda,\mu, \nu)=(2\pi)^{-\frac{1}{2}}\int_{\mathbb{R}}\int_{\mathbb{C}^n}F(t, w)\Phi_{\mu,\nu}(w)e^{i\lambda t}\,dwdt, \quad \forall \mu , \nu\in \mathbb{N}_0^n, \lambda\in \mathbb{R}.
\end{align} 
If $F \in L^2(\mathbb{R}\times\mathbb{C}^n)$, then $\hat{F} \in L^2(\mathbb{R}\times\mathbb{N}_0^{2n}, d\lambda\times d\sigma)$, where $d\lambda$ and $d\sigma$ denotes the Lebesgue measure on $\mathbb{R}$ and the counting measure on $\mathbb{N}_0^{2n}$, respectively. The
Plancherel formula is of the form $$\|F\|_{L^2(\mathbb{R}\times\mathbb{C}^n)} = \|\hat{F}(\lambda,\mu,\nu)\|_{L^2(\mathbb{R}\times\mathbb{N}_0^{2n})}$$
and the inverse Fourier–special Hermite transform is given by
$$F(t,z)= (2\pi)^{-\frac{1}{2}}\int_{\mathbb{R}} \sum_{(\mu, \nu) \in\mathbb{N}^{2n}_0}\hat{F}(\lambda, \mu, \nu)\Phi_{\mu,\nu}(z) e^{-i t \lambda}d\lambda.$$

 Given a surface $S$ in $\mathbb{R} \times \mathbb{N}_0^{2n}$ with a positive measure $d\Sigma$, we define the restriction operator $(\mathcal{R}_SF):=\{\hat{F}(\lambda,\mu, \nu)\}_{(\lambda,\mu, \nu)\in S}$  and the operator dual to $\mathcal{R}_S$ (called the extension operator)  as
$$\mathcal{E}_S(\{\hat{F}(\lambda,\mu, \nu)\})(t,z) :=(2\pi)^{-\frac{1}{2}}\int_{ S}e^{-i t \lambda} \hat{F}(\lambda,\mu, \nu) \Phi_{\mu,\nu} (z ) d\Sigma.$$
We consider the following special Hermite restriction problem: Let $I \subset \mathbb{R}$ be an interval.

\noindent{\bf Problem 1:} For which exponents $1\leq p, q\leq 2,$ the Fourier-Hermite restriction operator $\mathcal{R}_S$ is bounded from $L^p(I, L^q(\mathbb{C}^n))$ to $L^2(S, d\Sigma)$, i. e.,
$$\|\mathcal{R}_S f\|_{L^2(S, d\Sigma)}\leq \|f\|_{L^p(I, L^q(\mathbb{C}^n))} ?$$
By duality, it is not difficult to prove that the boundedness of $\mathcal{R}_S$ from $ L^p(I, L^q(\mathbb{C}^n))$ to $L^2(S, d\sigma)$  is equivalent to the boundedness of the operator $ \mathcal{E}_S(\mathcal{E}_S)^*$ from  $L^p(I, L^q(\mathbb{C}^n))$ to $L^{p’}(I, L^{q’}(\mathbb{C}^n))$, where $\frac{1}{p} + \frac{1}{p'} = 1$.
Therefore, Problem 1 can be re-written as follows:
\vspace{2pt}\\
\noindent{\bf Problem 2:} For which exponents  $1\leq p, q\leq 2,$  the operator  $\mathcal{E}_S(\mathcal{E}_S)^*$ is bounded from $ L^p(I, L^q(\mathbb{C}^n))$ to $ L^{p'}(I, L^{q'}(\mathbb{C}^n))$
$$\|\mathcal{E}_S(\mathcal{E}_S)^* f\|_{L^{p’}(I, L^{q’}(\mathbb{C}^n))}\leq \|f\|_{L^p(I, L^q(\mathbb{C}^n))} ?$$

   Again, by H\"{o}lder's inequality, the boundedness of $\mathcal{E}_S(\mathcal{E}_S)^*$  from $ L^p(I, L^q(\mathbb{C}^n))$ to $ L^{p'}(I, L^{q'}(\mathbb{C}^n))$ is equivalent to the following: for any $W_1, W_2 \in L^{{2p}/{(2-p)}}(I, L^{{2q}/{(2-q)}}(\mathbb{C}^n))$, the operator $W_1T_SW_2$ is bounded from $L^2(I\times \mathbb{C}^n)$ to $L^2(I\times\mathbb{C}^n)$ with the estimate
  \begin{align}\label{IN12}
  	\hspace{-8pt}\|W_1\mathcal{E}_S(\mathcal{E}_S)^*W_2\|_{L^2(I\times \mathbb{C}^n) \rightarrow L^2(I\times \mathbb{C}^n)} \hspace{-2pt}\leq \hspace{-2pt} C \|W_1\|_{L^{{2p}/{(2-p)}}(I, L^{{2q}/{(2-q)}}(\mathbb{C}^n))} \|W_2\|_{L^{{2p}/{(2-p)}}(I, L^{{2q}/{(2-q)}}(\mathbb{C}^n))}
  \end{align}
with $C > 0$ independent of $W_1, W_2$.

For certain surface $S$, the authors in \cite{SG} proved a stronger result than the estimate in (\ref{IN12}), proving that the operator $W_1\mathcal{E}_S(\mathcal{E}_S)^*W_2$  belongs to Schatten class $\mathcal{G}^{\alpha}(L^2((-\pi,\pi)\times \mathbb{C}^n))$, for some $\alpha > 0$ (see Subsection \ref{subsec2.2} for defination of Schatten class). More precisely,
 \begin{theorem}{\cite{SG}}
 Let $n,p,q \geq1$ and let the surface $S=\{(\lambda, \mu, \nu)\in \mathbb{R}\times\mathbb{N}_0^n\times\mathbb{N}_0^n : \lambda=2|\nu|+n\}$ with respect to the counting measure. Suppose $$ q > 2n+1 \quad\mbox{and}\quad \frac{2}{p} + \frac{2n}{q} = 1,$$
 then the inequality 
 \begin{align}
 	\|W_1T_{S}W_2\|_{\mathcal{G}^{q}(L^2((-\pi,\pi)\times \mathbb{C}^n))} \leq C \|W_1\|_{L^{p}((-\pi,\pi), L^{q}(\mathbb{C}^n))} \|W_2\|_{L^{p}((-\pi,\pi), L^{q}(\mathbb{C}^n))}
 \end{align}
holds for all $W_1, W_2$ with a constant $C > 0$ independent of $W_1,W_2$.
 \end{theorem}
\noindent Applying the duality principle  the following orthonormal Strichartz estimate for the Schr\"odinger semigroup associated to the special Hermite operator is obtaind in \cite{SG}.
\begin{theorem} \cite{SG}\label{main1}
	Let $n \geq 1$ and let $p, q \geq 1$ satisfies  the condition   $ 1\leq q < \frac{2n+1}{2n-1}\text { and } \frac{1}{p}+\frac{n}{q}=n.$ Then for any orthonormal system $(f_j)_j$ in $L^{2}\left(\mathbb{C}^n\right)$ and all sequence $ (n_{j })_j \subset \mathbb{C}$, we have
	\begin{align}\label{CH0A1}
		\left\|\sum_{j\in J} n_{j }\left|e^{-i t \mathcal{L}}f_{j }\right|^{2}\right\|_{L^p ((-\pi,\pi), L^q(\mathbb{C}^n))} \leqslant C \left(\sum_{j \in J}\left|n_{j}\right|^{\beta}\right)^{1/\beta},
	\end{align}
	where $C>0$ is independent of  $(f_j)_j$ and $(n_{j })_j $, with $\beta = \frac{2q}{q+1}$.
\end{theorem}

We now introduce a more general discrete surface in order to derive a localized restriction estimate.
 Let $\phi:\mathbb{R}^+\to \mathbb{R}$ be a smooth function such that there exists $0<m\leq 1$ such that 
\begin{equation}\label{e1}
\phi'(r)\sim r^{m-1} \quad\mbox{and}\quad |\phi''(r)|\geq r^{m-2}, ~r\geq 1.
\end{equation} 
Consider the discrete surface $S_{\phi}=\{(\lambda, \mu, \nu)\in \mathbb{R}\times\mathbb{N}_0^n\times\mathbb{N}_0^n : \lambda=\phi(2|\nu|+n)\}$ with respect to the localized measure $d\Sigma_N$, $N\geq 2n $ defined by
\begin{align}\label{az1}
	\int_{S_{\phi}} F\left(\lambda,\mu,\nu\right) d\Sigma_N = \sum_{\mu,\nu \in \mathbb{N}_0^n} {F(\phi(2|\nu|+n),\mu,\nu)}\psi\left(\frac{\sqrt{2|\nu|+n}}{N}\right),
\end{align} where  $\psi\in C_0^\infty(\mathbb{R})$ such that $\chi_{[-1,1]}\leq\psi\leq\chi_{[-2,2]},$ where $\chi_A$ denotes the characteristic function of the set $A\subset \mathbb{R}$. 

Let  $\mathcal{E}_{S_{\phi,N}}$ be the corresponding extension operator. We obtain the following Schatten estimate. 
\begin{theorem}\label{MAIN}
	Let $n,p,q \geq1$ and $N\geq 2n$. Let $\phi:\mathbb{R}^+\to \mathbb{R}$ be a smooth function satisfies (\ref{e1}) with $0< m \leq 1$. Let $S_{\phi}=\{(\lambda, \mu, \nu)\in \mathbb{R}\times\mathbb{N}_0^n\times\mathbb{N}_0^n : \lambda=\phi(2|\nu|+n)\}$ with respect to the measure $d\Sigma_N$. Suppose $$ q>2n \quad\mbox{and}\quad \frac{2}{p} + \frac{2n-1}{q} = 1 $$
 then the inequality 
 \begin{align}\label{Sch1}
 	\|W_1\mathcal{E}_{S_{\phi,N}}\mathcal{E}_{S_{\phi,N}}^*W_2\|_{\mathcal{G}^{q}(L^2((-T_0, T_0)\times \mathbb{C}^n))} \leq C N^\sigma \|W_1\|_{L^{p}((-T_0,T_0), L^{q}(\mathbb{C}^n))} \|W_2\|_{L^{p}((T_0,T_0), L^{q}(\mathbb{C}^n))}
 \end{align}
holds for all $W_1, W_2$ with a constant $C > 0$ independent of $W_1,W_2$, with $\sigma=\frac{2(2n(1-m)+m)}{q}.$
\end{theorem}
\noindent As a consequence, we derive the orthonormal Strichartz estimate for a more general class of semigroups associated with the special Hermite operator in the following theorem.
\begin{theorem}\label{MAIN2}
		Let $n,p,q \geq1$ and $N\geq 2n$. Let $\phi:\mathbb{R}^+\to \mathbb{R}$ be a smooth function satisfies (\ref{e1}) with $0< m \leq 1$. Suppose $p, q $ satisfy  the conditions $1\leq q <\frac{n}{n-1},~ \frac{2}{p}+\frac{2n - 1}{q} = 2n-1,$ and let $\sigma=(2n(1-m)+m)(1-\frac{1}{q})$. Then 
		\begin{align}\label{OSinq2}
			\left\|\sum_{j\in J} n_{j }\left|e^{-i t \phi(\mathcal{L})}f_{j }\right|^{2}\right\|_{L^p ((-T_0,T_0), L^q(\mathbb{C}^n))} \leqslant C N^\sigma\left(\sum_{j \in J}\left|n_{j}\right|^{\beta}\right)^{1/\beta},
		\end{align}
		holds for all orthonormal system $(f_j)_j$ in $L^2(\mathbb{C}^n)$ with $supp~ \hat{f}_j \subset \{(\mu,\nu): \sqrt{2|\nu|+n }\leq N\}$,  and all sequence $ (n_{j })_j \subset \mathbb{C}$ with $\beta = \frac{2q}{q+1}$.
\end{theorem}

Using the vector-valued version of the Littlewood–Paley inequality (see \cite{33}, Lemma 1), the frequency-localized estimate (\ref{OSinq2}) can be extended to a frequency-global estimate: 

\begin{theorem}\label{MAIN3}
		Let $n,p,q \geq1$ and $N\geq 2n$. Let $\phi:\mathbb{R}^+\to \mathbb{R}$ be a smooth function satisfies (\ref{e1}) with $0< m \leq 1$. Suppose $p, q $ satisfy  the conditions $1\leq q <\frac{n}{n-1},~ \frac{2}{p}+\frac{2n - 1}{q} = 2n-1,$ and let $s>\frac{2n(1-m)+m}{2}(1-\frac{1}{q})$. Then 
		\begin{align}\label{111}
			\left\|\sum_{j\in J} n_{j }\left|e^{-i t \phi(\mathcal{L})}f_{j }\right|^{2}\right\|_{L^p ((-T_0,T_0), L^q(\mathbb{C}^n))} \leqslant C \left(\sum_{j \in J}\left|n_{j}\right|^{\beta}\right)^{1/\beta},
		\end{align}
		holds for all orthonormal system $(f_j)_j$ in the Sobolev space $H^s({\mathcal{L}}) := \bigl\{ f \in L^2(\mathbb{C}^n) : \mathcal{L}^{\frac{s}{2}} f \in L^2(\mathbb{C}^n) \bigr\}$  and all sequence $ (n_{j })_j \subset \mathbb{C}$,  $\beta = \frac{2q}{q+1}$.
\end{theorem}

As an application of above Theorems, we obtain Strichartz estimates for orthonormal systems corresponding to the wave, Klein–Gordon, and fractional Schr\"{o}dinger equations associated with the special Hermite operator (see Section \ref{sec4}).


Returning to Theorem \ref{main1}, it was shown in \cite{SG}, using a semi-classical argument based on coherent states, that the Schatten exponent $\beta=\frac{2q}{q+1}$ in (\ref{CH0A1}) is optimal. But, the validity of the estimate (\ref{CH0A1}) remains unknown at the endpoint $(q, \beta)= (\frac{2n+1}{2n-1}, \frac{2q}{q+1}),$ which we answer negatively in the following theorem.

\begin{theorem}\label{end}
	The estimate (\ref{CH0A1}) fails at the endpoint $(q, \beta)= (\frac{2n+1}{2n-1}, \frac{2q}{q+1})$.
\end{theorem}

The next part of the paper focuses on generalizing the restriction estimate for the special Hermite spectral projections on the context of trace ideals. 
For $k\in\mathbb{N}_0,$ let $\mathcal{Q}_k$ denote the spectral projection operator on the eigenspace that corresponds to the $k$th eigenvalue $2k + n$ (see Section \ref{sec2}).
The following restriction estimate for the special Hermite spectral projections is well-known in the literature.
\begin{theorem}
	Let $n \geq 1$ and $k \geq 1$. Then 
	\begin{equation}\label{A2}
		\|\mathcal{Q}_k f\|_{2} \leq C_p k^{\varrho(p)}\|f\|_p
	\end{equation}
	holds with the exponent $\varrho(p)$ is given by
	\begin{equation}\label{rho}
		\varrho(p) =
		\begin{cases}
			n(1/p - 1/2) - 1/2, & \text{if } 1 \leq p \leq \frac{2(2n+1)}{2n +3} , \\ \vspace{2pt}
			-\frac{1}{2}(1/p - 1/2),  & \text{if } \frac{2(2n+1)}{2n +3}  \leq p \leq 2,
			
		\end{cases}
	\end{equation}
	and the estimate {\rm (\ref{A2})} is optimal in the sense that the exponent $\varrho(p)$ cannot be improved.
\end{theorem}

The estimate (\ref{A2}) 
was first established by Thangavelu \cite{Thanga, Than1} for the range 
$1 \leq p \leq \frac{2n}{n + 1}.$ 
Subsequently, Ratnakumar, Rawat, and Thangavelu \cite{Ratna1} 
extended this range to 
$1 \leq p < \frac{2(3n + 1)}{3n + 4}$. 
Later, Stempak and Zienkiewicz \cite{Stem} proved (\ref{A2}) 
for all $1 \leq p \leq 2$ except for  
$p = \frac{2(2n + 1)}{2n + 3}$. 
Finally, Koch and Ricci \cite{Koch}, settled the endpoint case 
$p = \frac{2(2n + 1)}{2n + 3}$, and showed that the estimate (\ref{A2}) is optimal. A local version of this endpoint estimate was obtained earlier by Thangavelu \cite{Than3}.

\vspace{.2cm}

Using duality argument, one can show that (\ref{A2}) is equivalent to 
\begin{equation}\label{1A2}
	\|\mathcal{Q}_k f\|_{p'} \leq C k^{2\varrho(p)}\|f\|_p,
\end{equation}
where $\frac{1}{p} + \frac{1}{p’} = 1$. By H\"older's inequality (\ref{1A2}) holds if and only if for any $W_1, W_2 \in {L^{2p/(2-p)}(\mathbb{C}^n)}$, the operator $W_1 \mathcal{Q}_k W_2$ is bounded on $L^2(\mathbb{C}^n)$ with the estimate 
\begin{equation}\label{2A2}
	\|W_1\mathcal{Q}_k W_2 \|_{L^2(\mathbb{C}^n) \rightarrow L^2(\mathbb{C}^n)} \leq C k^{2\varrho(p)}\| W_1\|_{L^{2p/(2-p)}(\mathbb{C}^n)} \|W_2\|_{L^{2p/(2-p)}(\mathbb{C}^n)},
\end{equation}
with $C>0$ independent of $W_1, W_2$.
\\

We upgrade the restriction estimate (\ref{2A2}) in the context of Schatten speace  $\mathcal{G}^\alpha(L^2(\mathbb{C}^n))$, as stated in the following theorem.

\begin{theorem}\label{Thm1}
	Let $n \geq 1$ and $k \geq 1$. Then for any $1 \leq p < \frac{2(3n + 1)}{3n + 4}$, there exists $C > 0$ such that for all $W_1, W_2 \in {L^{2p/(2-p)}(\mathbb{C}^n)}$, we have 
	\begin{align}\label{Q1}
		\left\|W_1 \mathcal{Q}_k {W_2}\right\|_{\mathcal{G}^{\frac{(3n -2)p}{6n - 1 - (3n+1)p}}\left(L^{2}\left( \mathbb{C}^{n} \right)\right)} \leq Ck^{2\varrho(p)} \| W_1\|_{L^{2p/(2-p)}(\mathbb{C}^n)} \|W_2\|_{L^{2p/(2-p)}(\mathbb{C}^n)},
	\end{align}
	where $\varrho(p)$ is defined in (\ref{rho}). Moreover, the exponent $\varrho(p)$  in (\ref{Q1}) is optimal.
\end{theorem}

By applying the duality principle (see Lemma 3 in \cite{frank1}) to (\ref{Q1}),  
 we immediately deduce the following result for orthonormal systems.
\begin{theorem}\label{Thm2.0}
	Let $n \geq 1$, $k \geq 1$, and $\alpha \geq 1$. Suppose $\frac{3n+1}{3n-2} <  p \leq \infty$,	then for any orthonormal system $\left(f_{j}\right)_{j \in J}$
	in  $L^2(\mathbb{C}^n)$ and any sequence $\left(n_{j}\right)_{j \in J} \subset \mathbb{C}$,  the estimate
	\begin{align}\label{CH1}
		\left\|\sum_{j \in J} n_{j} \left| \mathcal{Q}_k f_{j}\right|^{2}\right\|_{L^{p}(\mathbb{C}^{n})} \leq C k^{n(1-1/p) -1} \left(\sum_{j \in J}\left|n_{j}\right|^{\beta}\right)^{1/\beta},
		\end{align}
		where $C > 0$, with $\beta = \frac{2p(3n-2)}{6n-1}$. Moreover, the exponent $n(1-1/p) -1$  in (\ref{CH1}) is optimal.
\end{theorem}
The estimate (\ref{CH1}) reduces to (the dual of) (\ref{A2}) when the orthonormal system is reduced to one function (and the corresponding coefficient $n = 1$). We expect the Schatten exponent in (\ref{Q1}) can be improved to  $\frac{(2n-1)p}{4n-(2n+1)p}$. However, we obtain such an estimate only locally, in the following theorem.

\begin{theorem}\label{Thm3}
	Let $B$ be a fixed compact subset of $\mathbb{C}^n$. Let $n \geq 1$, $k \geq 1$ and $1\leq p \leq 2$. Then there exists $C_B > 0$ depends only on $B$ for all $W_1, W_2 \in {L^{2p/(2-p)}(\mathbb{C}^n)}$, we have the estimate
	\begin{align}\label{Q2}
		\left\|W_1 \chi_B \mathcal{Q}_k \chi_B {W_2}\right\|_{\mathcal{G}^{\alpha(p)}\left(L^{2}\left( \mathbb{C}^{n} \right)\right)} \leq C_B k^{2\varrho(p)} \| W_1\|_{L^{2p/(2-p)}(\mathbb{C}^n)} \|W_2\|_{L^{2p/(2-p)}(\mathbb{C}^n)},
	\end{align}
		where $\varrho(p)$ is defined in (\ref{rho}) and  the exponent $\alpha(p)$ is given by 
	\begin{equation}
			\alpha(p) =
		\begin{cases}
			\frac{(2n-1)p}{4n-(2n+1)p}, & \text{if } 1 \leq p \leq \frac{2(2n+1)}{2n +3} , \\ \vspace{4pt}
			\hspace{2pt}\frac{2p}{2-p},  & \text{if } \frac{2(2n+1)}{2n +3}  \leq p \leq 2.
		\end{cases}
	\end{equation}
   Moreover, the exponent $\varrho(p)$  in (\ref{Q2}) is optimal.
\end{theorem}

Again, by the duality principle, (\ref{Q2}) is equivalent to the following local version of the restriction estimate for orthonormal systems.
\begin{theorem}\label{Thm2}
	Let $B$ be a fixed compact subset of $\mathbb{C}^n$. Let $n \geq 1$, $k \geq 1$, $p \geq 1$ and $\alpha \geq 1$. Then for any orthonormal system $\left(f_{j}\right)_{j \in J}$
	in  $L^2(\mathbb{C}^n)$ and any sequence $\left(n_{j}\right)_{j \in J} \subset \mathbb{C}$,  the estimate
	\begin{align}\label{CH12}
		\left\|\sum_{j \in J} n_{j} \left| \mathcal{Q}_k\chi_B f_{j}\right|^{2}\right\|_{L^{p}(\mathbb{C}^{n})} \leq C_B k^{\varrho'(p)} \left(\sum_{j \in J}\left|n_{j}\right|^{\beta}\right)^{1/\beta},
	\end{align}
	holds for $\beta = \frac{2p}{p+1}$ for $1 \leq p \leq \frac{2n+1}{2n -1}$ and $\beta = \frac{p(2n-1)}{2n}$ for $\frac{2n+1}{2n -1} \leq p \leq \infty$, with 
	\begin{equation}\label{rho1}
		\varrho'(p) =
		\begin{cases}
			-\frac{1}{2}(1-1/p),  & \text{if } 1  \leq p \leq \frac{2n+1}{2n -1}, \\ \vspace{2pt}
			n(1-1/p) - 1, & \text{if } \frac{2n+1}{2n -1} \leq p \leq \infty. 
		\end{cases}
	\end{equation}
	Moreover, the exponent $\varrho'(p)$  in (\ref{CH12}) is optimal.
\end{theorem}
We believe that Theorem \ref{Thm2} remains valid globally, and that the Schatten exponent $\beta$ given in Theorem \ref{Thm2} is optimal. 

The structure of the paper, apart from the introduction, is as follows. In Section \ref{sec2}, we recall the necessary background and discuss dispersive estimates for certain semigroups associated with $\mathcal{L}$. In Section \ref{sec3},  we prove Theorems \ref{MAIN}, \ref{MAIN2} and obtain Strichartz estimates for orthonormal systems corresponding to the wave, Klein–Gordon, and fractional Schr\"{o}dinger equations associated with $\mathcal{L}$. In Section \ref{sec4}, we show that the Strichartz estimates for orthonormal systems associated with the Schr\"{o}dinger propagator for $\mathcal{L}$ fail at the endpoint. Finally, in Section \ref{sec5}, we obtain the results concerning the restriction theorem for the special Hermite spectral projections.

\section{Preliminaries}\label{sec2}
In this section we provide some basic definitions and discuss certain dispersive semigroups associated with the special Hermite operator.
\subsection{Hermite operator and special Hermite functions} 
The Hermite polynomials on $\mathbb{R}$ are defined by
 $$H_k(x)=(-1)^k \frac{d^k}{dx^k}(e^{-x^2} )e^{x^2}, \quad k\in \mathbb{N}_0$$
 and  the normalized Hermite functions $h_k$ on $\mathbb{R}$ are defined by
 $$h_k(x)=(2^k\sqrt{\pi} k!)^{-\frac{1}{2}} H_k(x)e^{-\frac{1}{2}x^2}.$$ 
 The higher dimensional Hermite functions  $\Phi_{\alpha}$ are obtained by taking tensor product of one dimensional Hermite functions. For any multi-index $\alpha \in \mathbb{N}_0^n$ and $x=(x_1, x_2\cdots,x_n) \in \mathbb{R}^n$, we define
 $\Phi_{\alpha}(x)=\prod_{j=1}^{n}h_{\alpha_j}(x_j).$
 The special Hermite functions are then defined as the Fourier-Wigner transform of the Hermite functions $\Phi_\mu$ and $\Phi_\nu$ on $\mathbb{R}^n$, namely,
 $$\Phi_{\mu, \nu}(z)=(2 \pi)^{-\frac{n}{2}} \int_{\mathbb{R}^{n}} e^{i x \cdot \xi} \Phi_\mu\left(\xi+\frac{y}{2}\right)  \Phi_\nu\left(\xi-\frac{y}{2}\right) d \xi, \quad z=x+i y \in \mathbb{C}^{n}.$$  The family of functions $\{\Phi_{\mu, \nu}\}$ forms an orthonormal basis for $L^2(\mathbb{C}^n)$  leading to the special Hermite expansion
 \begin{align}\label{tyu}
 	f(z)=\sum_{\mu ,\nu\in \mathbb{N}^{n}_0}\langle f, \Phi_{\mu, \nu}\rangle  \Phi_{\mu, \nu}(z) = \sum_{k=0}^{\infty} \left(\sum_{|\nu|=k}\sum_{\mu\in \mathbb{N}^{n}_0}\langle f, \Phi_{\mu, \nu}\rangle  \Phi_{\mu, \nu}(z) \right).
 \end{align}
 
 The spectral projection operator $\mathcal{Q}_k$ onto the eigenspace of $\mathcal{L}$ associated to the eigenvalue $2k + n$ is given by 
 \begin{align}
 	\mathcal{Q}_k f = \sum_{|\nu|=k}\sum_{\mu\in \mathbb{N}^{n}_0}\langle f, \Phi_{\mu, \nu}\rangle  \Phi_{\mu, \nu}(z),  \quad f\in \mathcal{S}(\mathbb{C}^n).
 \end{align}
 For each \(k \in \mathbb{N}\),  the spectral decomposition of \(\mathcal{L}\) can be written as
$$
\mathcal{L} f=\sum_{k=0}^{\infty}(2 k+n) \mathcal{Q}_{k} f.
$$
The twisted convolution of two functions $f$ and $g$ on $\mathbb{C}^{n}$ is defined by
\begin{equation}\label{con}f \times g(z)=\int_{\mathbb{C}^{n}} f(z-w) g(w) e^{\frac{i}{2} \operatorname{Im}(z \cdot \bar{w})} d w=\int_{\mathbb{C}^{n}} f(w) g(z-w) e^{-\frac{i}{2} \operatorname{Im}(z \cdot \bar{w})} d w, \quad\mbox{$z\in \mathbb{C}^n$}.\end{equation}
The family $\{\Phi_{\mu, \nu} \}$ satisfies the following orthogonality properties
\begin{align}\label{CH3sh1}
	\Phi_{\mu, \nu} \times \Phi_{\alpha ,\beta}=\left\{\begin{array}{ll}{(2 \pi)^{n / 2} \Phi_{\mu, \beta} }, & {\text{if } \nu=\alpha,} \\ {0}, & {\text{otherwise}.} \end{array}\right.
\end{align}

The special Hermite functions $\Phi_{\nu, \nu}$  are related to the Laguerre functions $\varphi_{k}(z)=L_{k}^{n-1}\left(\frac{1}{2}|z|^{2}\right) e^{-\frac{1}{4}|z|^{2}}$, where $L_{k}^{n-1}$ is the Laguerre polynomial of type $(n-1),$  by the following relation
\begin{align}\label{CH33333}
	(2 \pi)^{n / 2} \sum_{|\nu|=k} \Phi_{\nu, \nu}=\varphi_{k}.
\end{align}
Now taking twisted convolution on both sides of (\ref{tyu}) with $\Phi_{\alpha, \alpha}$  and using the orthogonality property (\ref{CH3sh1}),  we have
\begin{align}\label{CH3333}
	f \times \Phi_{\alpha, \alpha}=(2 \pi)^{n / 2} \sum_{\mu \in \mathbb{N}^n_0}\left\langle f, \Phi_{\mu ,\alpha}\right\rangle \Phi_{\mu ,\alpha}, \quad \alpha \in \mathbb{N}^n_0.
\end{align}

Summing both sides of (\ref{CH3333}) with respect to all $\alpha$ such that $|\alpha|=k$ and using (\ref{CH33333}),   the spectral projection $\mathcal{Q}_{k}$  has the simpler representation
$$
\mathcal{Q}_{k} f(z)=(2 \pi)^{-\frac{n}{2}} \sum_{|\alpha|=k} f \times \Phi_{\alpha, \alpha}(z)=(2 \pi)^{-n} f \times {\varphi_{k}}(z), \quad z\in \mathbb{C}^n.
$$ 

For a detailed study on the special Hermite operator and its associated functions, we refer the reader to the monograph by Thangavelu \cite{Thanga}.

\subsection{Dispersive Semigroups Associated with the Special Hermite Operator}

Consider the Schr\"{o}dinger equation on $\mathbb{C}^n$:
\begin{align}\label{eq1}
	i \partial_t u(t, z)- \mathcal{L}u(t, z) &= 0,\quad z \in \mathbb{C}^n, \hspace{2pt} t \in \mathbb{R}\setminus \{0\}, \\
	\nonumber u(0,z) &= f(z),
\end{align}
If $f \in L^2(\mathbb{C}^n)$, then the solution to the IVP (\ref{eq1}), $e^{-i t \mathcal{L}} f$. The Schr\"odinger propagator can be expressed by using the spectral decomposition of $\mathcal{L}$, that is 
$$
e^{-it \mathcal{L}} f=\sum_{k=0}^{\infty} e^{-(2 k+n) it} \mathcal{Q}_{k}f.
$$
So, we clearly have 
$$\|e^{-it \mathcal{L}} f\|_{L^2(\mathbb{C}^n)} = \|f\|_{L^2(\mathbb{C}^n)}, \quad t \in \mathbb{R}.$$
The Schr\"odinger propagator $e^{-it \mathcal{L}}$ also has the following kernel representation:
\begin{equation}\label{KU}
	e^{-i t \mathcal{L}} f(z)=\int_{\mathbb{C}^{n}} f(z-w) K_{it}(w) e^{\frac{i}{2} \operatorname{Im} (z \cdot \bar{w})} d w,
\end{equation}
where the kernel is given by
$$K_{it}(w) = (2 \pi i \sin t)^{-n} e^{i \cot t \frac{|w|^2}{4}}.$$
This can be easily deduced from the corresponding kernel formula for the heat operator $e^{-t \mathcal{L}}$ by replacing $t$ with $it$ (see \cite{Thanga}, page 29). 
Let $\phi:\mathbb{R}^+\to \mathbb{R}$ be a smooth function satisfies (\ref{e1}) with $0< m \leq 1$. Consider the following Schr\"{o}dinger equation on $\mathbb{C}^n$:
\begin{align}\label{eq1}
	i \partial_t u(t, z)-{\phi(\mathcal{L})}u(t, z) &= 0,\quad z \in \mathbb{C}^n, \hspace{2pt} t \in \mathbb{R}\setminus \{0\}, \\
	\nonumber u(0,z) &= f(z).
\end{align}
If $f \in L^2(\mathbb{C}^n)$, then the solution to the IVP (\ref{eq1}), is given by $u(t,z)=e^{-i\, t\, {\phi(\mathcal{L})}} f(z)$,
where
\begin{align}\label{L}
e^{-i\, t\, {\phi(\mathcal{L})}} f=\sum_{k=0}^{\infty} e^{-i\,t\,\phi(2 k+n) }  \mathcal{Q}_{k}f.
\end{align}
Since $\mathcal{L}$ satisfies conditions (A1) and (A2), and $\phi$ satisfies condition (H1) in Theorem 3.3 of \cite{BUI} (see also \cite{FM}), it follows that the following local dispersive estimate holds:
\begin{align}\label{L1}
\left\|\psi(h^{-1}\sqrt{\mathcal{L})}e^{-it {\phi(\mathcal{L})}}f\right\|_{L^\infty(\mathbb{C}^n)} \leq |t|^{-\frac{2n-1}{2}}h^{(1-m)2n+m}\left\|f\right\|_1,
\end{align} where $\psi\in C^\infty(\mathbb{R})$ supported in $[1/2,2]$, $h\geq 1$ and $|t|\leq T_0<\pi.$

\subsection{Schatten class}\label{subsec2.2}
 Let $X$ be a measure space. Let $A:L^2(X) \rightarrow L^2(X)$ be a compact operator and let    $A^{*}$ denotes the adjoint of $A$. For $1 \leq r<\infty,$ $ \mathcal{G}^{r}({L^2(X)})$ denotes the Schatten space based on $L^2(X)$ that is the space of all compact operators $A$ on $L^2(X)$ such that $Tr |A|^r < \infty$, where \(|A|=\sqrt{A^{*} A}\), and its norm is defined by
$\|A\|_{\mathcal{G}^{r}(L^2(X))}= (Tr |A|^r)^{\frac{1}{r}}$. If $r = \infty $, we define $$\|A\|_{\mathcal{G}^{\infty}(L^2(X))} = \|A\|_{L^2(X) \rightarrow L^2(X)}.$$
Also, the case $r = 2 $ is special in the sense that $\mathcal{G}^2(L^2(X))$ is the Hilbert-Schmidt class, equipped  with the norm 
$$\|A\|_{\mathcal{G}^{2}(L^2(X))} = \|K\|_{L^2(X \times X)},$$ if $K$ is the integral kernel of $A$. For more details on Schatten classes, we refer the reader to Simon \cite{Simon}.

 \section{Strichartz inequality for system of  orthonormal functions}\label{sec3}
This section is devoted to the proofs of Theorems \ref{MAIN} and \ref{MAIN2}. As an application, we obtain Strichartz estimates for orthonormal systems corresponding to the wave, Klein–Gordon, and fractional Schr\"{o}dinger equations associated with the special Hermite operator.

In order to obtain the desired estimate for a system of orthonormal functions we need a duality principle lemma in our setting. By referring to Lemma 3 of \cite{frank1} and applying necessary modifications, we obtain the following result:

 \begin{lem}\label{CH3dual}(Duality principle)
 	Let $\mathcal{H}$ be a separable Hilbert space. Suppose $A$ is a bounded linear operator from $\mathcal{H}$ to $L^{p'}(I, L^{q'}(\mathbb{C}^n))$, where $1 \leq p, q \leq 2$, $I \subset \mathbb{R}$ and let $\alpha \geq 1$. Then the following statements are equivalent:

 	\begin{enumerate}
 		\item There is a constant \(C>0\) such that
 		\begin{align}\label{CH3511}
 			\left\|W A A^{*} \overline{W}\right\|_{\mathcal{G}^{\alpha}\left(L^{2}\left( I \times\mathbb{C}^{n} \right)\right)} \leq C\|W\|_{L^{\frac{2p}{2-p}}(I, L^{\frac{2q}{2-q}}(\mathbb{C}^n))}^{2} \end{align} for all $W \in {L^{\frac{2p}{2-p}}(I, L^{\frac{2q}{2-q}}(\mathbb{C}^n))}$, where the function $W$ is interpreted as an operator which acts by multiplication.
 		
 		\item  For any orthonormal system $\left(f_{j}\right)_{j \in J}$
 		in  $L^2(\mathbb{C}^n)$ and any sequence $\left(n_{j}\right)_{j \in J} \subset \mathbb{C}$, there is a constant \(C'>0\) such that
 		\begin{align}\label{CH3512}
 			\left\|\sum_{j \in J} n_{j} \left| A f_{j}\right|^{2}\right\|_{L^{\frac{p'}{2}}(I, L^{\frac{q'}{2}}(\mathbb{C}^n))} \leq C' \left(\sum_{j \in J}\left|n_{j}\right|^{\alpha^{\prime}}\right)^{1 / \alpha^{\prime}}.
 		\end{align}
 	\end{enumerate}
 \end{lem}

Now, we proceed to the proof of Theorem \ref{MAIN}. Consider $S_\phi=\{(\lambda, \mu, \nu)\in \mathbb{R}\times\mathbb{N}_0^n\times\mathbb{N}_0^n : \lambda=\phi(2|\nu|+n)\}$, with respect to the measure $d\sigma_N$ defined in (\ref{az1}). Then
\begin{align}\label{aaone}\nonumber
	\mathcal{E}_{S_{\phi,N}}\mathcal{E}_{S_{\phi,N}}^* f (t, z)&=  \sum_{(\lambda, \mu,\nu)\in S_\phi} \psi^2\left(\sqrt{\lambda}/{N}\right){e^{-it\lambda}} \hat{f}(\lambda, \mu,\nu)\Phi_{\mu,\nu}(z)\nonumber\\
	&=(2\pi)^{-\frac{1}{2}} \sum_{(\mu,\nu,\lambda)\in S_\phi}  \int_{\mathbb{R}} \psi^2\left(\sqrt{\lambda}/{N}\right){e^{-it\lambda}}  \langle {f}(t', \cdot), \Phi_{\mu,\nu}\rangle   \Phi_{\mu,\nu}(z)  ~dw~dt'\\
	&=(2\pi)^{-\frac{1}{2}}   \int_{\mathbb{R}}\sum_{\mu, \nu} \psi^2\left(\frac{\sqrt{2|\nu|+n}}{N}\right) {e^{-i(t-t')\phi(2|\nu|+n)} } \langle {f}(t', \cdot), \Phi_{\mu,\nu}\rangle   \Phi_{\mu,\nu}(z)   ~dw~dt'.
\end{align}
Now 	from  \eqref{CH3333} and (\ref{CH33333}), we  write
\begin{align} \nonumber
	\mathcal{E}_{S_{\phi,N}}\mathcal{E}_{S_{\phi,N}}^* f (t, z)&=(2\pi)^{-{(n+\frac{1}{2})}}   \int_{\mathbb{R}} {f}(t', \cdot )\times  \sum_{ k=0}^\infty  \psi^2\left(\frac{\sqrt{2k+n}}{N}\right) {e^{-i(t-t')\phi(2k+n)} }\varphi_{k}(z)dt' \\\nonumber
	&=  (2\pi)^{-{(n+\frac{1}{2})}}       \int_{\mathbb{R}}   \int_{\mathbb{C}^n}   H_{\phi,N}(t-t', z-w) {f}( t', w) e^{-\frac{i}{2} \operatorname{Im} (z\cdot \bar{w})} dt' dw
\end{align}
where \begin{align}\label{new kernel}
	H_{\phi, N}(t, z) =  \sum_{ k=0}^\infty  \psi^2\left(\frac{\sqrt{2k+n}}{N}\right) {e^{-it\phi(2k+n)} }\varphi_{k}(z).
\end{align}\\
\vspace{0pt}

\vspace{6pt}
\noindent{\bf \emph{Proof of Theorem \ref{MAIN}} : }
Let $\Psi \in C^\infty(\mathbb{R})$ be an even function such that $0 \leq \Psi \leq 1$, $\Psi = 1$ in $[0,1]$ and $\Psi =0$ in $[2,\infty)$. Let $\varPsi(s)=\Psi(s) - \Psi(2s)$ such that $supp \,\,\varPsi \subset (1/2,2)$ and generate a Littlewood-Paley decomposition $\sum_{j\in \mathbb{Z}} \varPsi_j(s) =1 $, for all $s > 0,$ where  $\varPsi_j(s) = \varPsi(2^{-j}s)$. It follows that  $$\sum_{j\geq 0} \varPsi_j(s) =1, \quad \forall s \geq 1.$$

	For any small $\varepsilon > 0$, $\alpha \in \mathbb{C}$ with $\operatorname{Re}\alpha \in\left[-r/2, 0\right]$ and $j \geq 0$, we consider the following analytic family of operators  
	
	$$G_{\alpha, j, \varepsilon}f(t,z)=     \int_{\mathbb{C}^n} \int_{I}  H_{\phi, N, j, \varepsilon} (t-t', z-w) e^{-\frac{i}{2} \operatorname{Im} (z \cdot \bar{w})} {f}( t', w)~ dt' dw,$$
	where the kernel
	$$
	H_{\phi, N, j, \varepsilon} (t, z)=1_{\varepsilon < |T_0|} \,t^{-1-\alpha}  \sum_{ k=0}^\infty  \psi^2\left(\frac{\sqrt{2k+n}}{N}\right)\varPsi_j\left(\sqrt{2k+n}\right) {e^{-it\phi(2k+n)} }\varphi_{k}(z).
	$$
	But from (\ref{L1}),   for  every $ t  \in [-T_0,T_0]=I$ and $ z \in  \mathbb{C}^n   $, we have
	\begin{align}\label{CH2two}
		\left|H_{\phi, N, j, \varepsilon} (t, z)\right| \leq C|t|^{-\frac{2\operatorname{Re} (\alpha)+2n + 1 }{2}} 2^{((1-m)2n + m)j}.
	\end{align}
Let $z_1 = -\frac{r}{2} + i \tau,  \tau \in \mathbb{R}$, using the Hardy–Littlewood–Sobolev inequality, (see page 39 in \cite{bek1}) along with (\ref{CH2two})  yields
	\begin{align*}
			\left\|W_1 G_{ z_1, \varepsilon}  W_2\right\|_{\mathcal{G}^{2}\left(L^{2}\left(I
			\times \mathbb{C}^{n}\right)\right)}^2 
		&=\int_{I\times I}\int_{\mathbb{C}^{2n}} |W_{1}(t, z)|^2 \left|H_{\varrho,z_1,\varepsilon} ( t-t', z-w)\right|^{2} |W_{2}\left(t^{\prime}, w\right)|^{2 } dz~ d w ~d t ~d t^{\prime}\\
		&\leq C 2^{((1-m)2n + m)2j}  \int_{I } \int_{I } \frac{\left\|W_{1}(t)\right\|_{L^{2  }\left(\mathbb{C}^{n}\right)}^{2  }\left\|W_{2}\left(t^{\prime}\right)\right\|_{L^{2  }\left(\mathbb{C}^{n}\right)}^{2  }}{\left|t-t^{\prime}\right|^{2n+1 -r}} d t d t^{\prime}\\
		&\leq C2^{((1-m)2n + m)2j}   \left\|  \|W_{1} \|^2_{L_{w}^{2  }(\mathbb{C}^n) }\right \|_{ L_{t}^{{u}} \left(I \right)}
		\left\|  \|W_{2} \|^2_{L_{w}^{2  }(\mathbb{C}^n) }\right \|_{ L_{t}^{{u}} \left(I \right)},
		\end{align*}
		provided we have  $0 \leq 2n + 1 -r  <1$ and $\frac{1}{u}+\frac{2n + 1 -r}{2} =1$. Thus, for $2n < r \leq 2n+1$, we obtain
		$$\left\|W_1 G_{ z_1, j, \varepsilon}  W_2\right\|_{\mathcal{G}^{2}\left(L^{2}\left(I
			\times \mathbb{C}^{n}\right)\right)} \leq C 2^{((1-m)2n + m)j}\|W_1\|_{L^\frac{4}{r-2n+1}(I, L^2\left(\mathbb{C}^{n}\right))}\|W_2\|_{L^\frac{4}{r-2n+1}(I, L^2\left(\mathbb{C}^{n}\right))},$$
			with the constant $C$ is independent of $\varepsilon$ and $\tau$.

   Next,  we   consider the case   $z_2 = i\tau, \tau \in \mathbb{R}$. We show  that $G_{z_2 ,j,\varepsilon}: L^2\left( I \times \mathbb{C}^{n} \right) \rightarrow L^2\left( I \times \mathbb{C}^{n} \right)$  is bounded  with some constant   that only depends on  the dimension $n$ and $\tau$ exponentially.  After a simple calculation, we write
	\begin{align*}
		G_{z_2, j, \varepsilon}f(t,z)
		=(2{\pi})^n\sum_{\mu, \nu}    \Phi_{\mu,\nu}(z)  \psi\left(\frac{\sqrt{2|\nu|+n}}{N}\right)&\varPsi_j\left(\sqrt{2|\nu|+n}\right)\\ \cdot &\int_{\varepsilon<|t'|< T_0}     {e^{-i t'(2|\nu|+n)^{\varrho}}}   {t'}^{-1- \alpha }  \hat{f}_2(t-t', \cdot) (\mu, \nu) dt',
	\end{align*}
	where $  \hat{f}_2$ denotes the special Hermite transform of $f$ with respect to the second variable. Then  using  Plancherel's theorem,  for each $t \in {I}$, we have
	\begin{align}\label{newdefine11}
	\left\|G_{z_2, j, \varepsilon} f(t, \cdot)\right\|_{L^2(\mathbb{C}^n)}^2 
		&\leq C  \sum_{\mu,\nu}    \left| \int_{\varepsilon<|t'|< T_0}   {t'}^{-1+i \tau} \varTheta_{\mu, \nu}\left(t-t'\right)~dt'\right|^2
	\end{align}
	where $\varTheta_{\mu,\nu}(t)=e^{-i t\phi(2|\nu|+n) }      \hat{f}_2( t, \cdot ) (\mu,\nu)$.  If we define $$\varGamma_{z_2, \varepsilon}: \varTheta(t) \mapsto\int_{\varepsilon<|t'|< T_0}   {t'}^{-1+i \tau} \varTheta\left(t-t'\right)\;dt',$$ then (\ref{newdefine11}) becomes
	
	\begin{align}
		\left\|G_{z_2, j, \varepsilon} f\right\|_{L^2( I \times \mathbb{C}^n)} \leq  (2{\pi})^{2n}   \sum_{\mu,\nu}    \left\| \varGamma_{z_2, \varepsilon} \varTheta_{\mu,\nu} \right\|_{L^2(I)}^2.
	\end{align}
	Since, the operator $\varGamma_{z_2, \varepsilon}$ is just a  Hilbert transform up to $i \tau $, from \cite{vega},   the operator $\varGamma_{z_2, \varepsilon}: L^2 \rightarrow L^2$ is bounded with constant   depends  only on $\tau$ exponentially. Thus, using  the boundedness of $$G_{z_2, j, \varepsilon} :  L^2\left(I \times \mathbb{C}^{n} \right) \rightarrow L^2\left(I \times \mathbb{C}^{n} \right)$$ and the fact that $S_{\infty}$-norm is the operator norm, we have 
	$$
	\left\|W_1 T_{z_2, \varepsilon}  W_2\right\|_{\mathcal{G}^{\infty}\left(L^{2}\left(I \times \mathbb{C}^{n}\right)\right)} \leq C (\tau) \left\|W_1\right\|_{L^\infty(I, L^\infty\left(\mathbb{C}^{n}\right))}\left\|W_2\right\|_{L^\infty(I, L^\infty\left(\mathbb{C}^{n}\right))}.	$$
	Applying Stein's analytic interpolation result \cite{ber, interpolation}, we get
	$$
	\left\|W_1 G_{-1,j, \varepsilon}  W_2\right\|_{\mathcal{G}^{r}\left(L^{2}\left(I \times \mathbb{C}^{n}\right)\right)} \leq C 2^{\frac{((1-m)2n + m)2j}{r}} \left\|W_1\right\|_{L^\frac{2r}{r-2n+1}(I, L^r\left(\mathbb{C}^{n}\right))}\left\|W_2\right\|_{L^\frac{2r}{r-2n+1}(I, L^r\left(\mathbb{C}^{n}\right))},$$
	with the constant independent of $\varepsilon$ and $j$. Summing over $2^j \leq N$ and letting $\varepsilon \rightarrow 0$, we obtain (\ref{Sch1}) for $2n < q \leq 2n+1$, since $\mathcal{E}_{S_{\phi,N}}\mathcal{E}_{S_{\phi,N}}^* = \sum_{\{j :\,2^j\leq N\}}G_{-1, j, \varepsilon} $ in the limit as $\varepsilon \rightarrow 0$.
	
	Again, it is easy to check that $$ \left\|\sum_{j \in J} n_{j} \left| \mathcal{E}_{S_{\phi,N}} f_j\right|^{2}\right\|_{L^{\infty}(I, L^{1}(\mathbb{C}^n))} \leq  \sum_{j \in J}\left|n_{j}\right| ~\mbox{sup}_{t\in I}~ \|\mathcal{E}_{S_{\phi,N}} f_j\|^2_{ L^2(\mathbb{C}^n))} \leq \sum_{j \in J}\left|n_{j}\right| .$$ Hence by duality Lemma \ref{CH3511}, we have  
	\begin{align}\label{label}
		\|W_1\mathcal{E}_{S_{\phi,N}}\mathcal{E}_{S_{\phi,N}}^*W_2\|_{\mathcal{G}^{\infty}(L^2(I\times \mathbb{C}^n))} \leq C\left\|W_1\right\|_{L^2(I, L^\infty\left(\mathbb{C}^{n}\right))}\left\|W_2\right\|_{L^2(I, L^\infty\left(\mathbb{C}^{n}\right))}.
	\end{align}
	Interpolating (\ref{Sch1}) with $q = 2n+1$ and (\ref{label}), we get (\ref{Sch1}) for $q \geq 2n+1$.
 $\hfill{\square}$	

\vspace{5pt}

\vspace{6pt}
\noindent{\bf \emph{Proof of Theorem \ref{MAIN2}} : }
	Consider  the discrete surface  $S_\phi= \{(\lambda, \mu, \nu)\in \mathbb{R} \times \mathbb{N}_0^n\times \mathbb{N}_0^n : \lambda=(2|\nu|+n)\}$ with respect to the measure $\sigma_N$ defined in (\ref{az1}). Then for all $f$ such that $\hat{f} \in \ell^{1}(S_\phi)$ and for all $(t, z)\in \mathbb{R}\times \mathbb{C}^n$, the extension operator can be written as
\begin{align}\label{CH3surfacee}
	\mathcal{E}_{S_{\phi,N}} f(t, z)=\sum_{(\lambda, \mu, \nu)\in S} \psi({\sqrt{\lambda}}/{N}){e^{-it\lambda}}\hat{f}(\lambda, \mu, \nu) \Phi_{\mu, \nu} (z),
\end{align}   where $\hat{f}(\lambda, \mu, \nu)$ is defined in (\ref{Fourier-special Hermite transform}). 	Using the fact that $$f\times\Phi_{\mu\mu}=(2\pi)^{\frac{n}{2}}\sum_{\nu}\langle f, \Phi_{\mu, \nu}\rangle  \Phi_{\mu, \nu},$$ and choosing
$$
\hat{f}(\lambda, \mu, \nu)=\left\{\begin{array}{ll}       {(2\pi)^n\langle u, \Phi_{\mu, \nu}\rangle }, & \text {if}~ \lambda=\phi(2|\nu|+n) \, \text{and} \,\sqrt{2|\nu|+n} \leq N\\{0}, & \text {otherwise}, \end{array} \right.
$$ for some $u:\mathbb{C}^n\to \mathbb{C}$
in (\ref{CH3surfacee}), we get
\begin{align}\label{nn}\nonumber
	\mathcal{E}_{S_{\phi,N}} f(t, z)&=(2\pi)^n\sum_{\nu} \left( \sum_{\mu}\langle u, \Phi_{\mu, \nu}\rangle  \Phi_{\mu, \nu}(z)\right) e^{-it \phi(2|\nu|+n)}\nonumber\\
	&=(2\pi)^{\frac{n}{2}}\sum_{\nu}  e^{-it \phi(2|\nu|+n)}u\times  \Phi_{\nu, \nu}(z)~\nonumber\\
	&=\sum_{k=0}^{\infty} e^{-it \phi(2k+n)} u\times   \varphi_k(z)=e^{-it\phi(\mathcal{L})}u(z).
\end{align}
Thus, Theorem \ref{MAIN2} follows from Theorem \ref{MAIN} together with the duality principle (Lemma \ref{CH3dual}). 
$\hfill \square$

\vspace{6pt}
\noindent{\bf \emph{Proof of Theorem \ref{MAIN3}} : } 
We adopt an approach based on \cite{nakamura}. Let $\varPhi \in C^{\infty}_c$ such that $supp\,\, \varPhi \subset [1/2,2]$ and $h \geq 1$.
Then Theorem \ref{MAIN2} can be rephrased as
\begin{align}\label{qw541}
			\left\|\sum_{j\in J} n_{j }\left|e^{-i t \phi(\mathcal{L})}\varPhi(h^{-1}\sqrt{\mathcal{L}})f_{j }\right|^{2}\right\|_{L^p ((-T_0,T_0), L^q(\mathbb{C}^n))} \leqslant C h^\sigma\left(\sum_{j \in J}\left|n_{j}\right|^{\beta}\right)^{1/\beta}.
		\end{align}

Let $\{\varPsi_\ell\}_{\ell\geq 0}$ be defined as in the proof of Theorem \ref{MAIN}.  
Using the vector-valued version of the Littlewood–Paley inequality (see \cite{33}, Lemma 1), triangle inequality, we obtain 
\begin{align*}
			\left\|\sum_{j\in J} n_{j }\left|e^{-i t \phi(\mathcal{L})}\mathcal{L}^{-\frac{\sigma}{4}-\varepsilon}f_{j }\right|^{2}\right\|_{L^p ((-T_0,T_0), L^q(\mathbb{C}^n))} &\leqslant  \sum_{\ell=0}^\infty\left\|\sum_{j\in J} n_{j }\left|\mathcal{L}^{-\frac{\sigma}{4}-\varepsilon} e^{-i t \phi(\mathcal{L})}\varPsi_\ell(\sqrt{\mathcal{L}})f_{j }\right|^{2}\right\|_{L^p ((-T_0,T_0), L^q(\mathbb{C}^n))}\\&= \sum_{\ell=0}^\infty2^{-\ell(\sigma+4\varepsilon)}\left\|\sum_{j\in J} n_{j }\left|e^{-i t \phi(\mathcal{L})}\tilde{\varPsi}_\ell(\sqrt{\mathcal{L}})f_{j }\right|^{2}\right\|_{L^p ((-T_0,T_0), L^q(\mathbb{C}^n))},
\end{align*}
where $\tilde{\varPsi}_\ell(s) = \tilde{\varPsi}(2^{-\ell} s)$ with  $\tilde{\varPsi}(s)= s^{-\frac{\sigma}{2}-2\varepsilon}\varPsi(s)$. Clearly, $\tilde{\varPsi}\in C^\infty_c(\mathbb{R})$ and $supp\, \tilde{\varPsi} \subset [1/2,2]$, thus from (\ref{qw541}) we obtain
\begin{align*}
\left\|\sum_{j\in J} n_{j }\left|e^{-i t \phi(\mathcal{L})}\mathcal{L}^{-\frac{\sigma}{4}-\varepsilon}f_{j }\right|^{2}\right\|_{L^p ((-T_0,T_0), L^q(\mathbb{C}^n))} \leq C_\varepsilon \left(\sum_{j \in J}\left|n_{j}\right|^{\beta}\right)^{1/\beta}.
\end{align*}
This completes the proof of the theorem. 
$\hfill\square$	
	
\vspace{6pt}		
\begin{rem}
In \cite{Feng1}, the authors established Strichartz estimates for systems of orthonormal functions associated with general dispersive flows of the form $e^{- i t \phi(L)} \psi(\sqrt{L})$, where $L$ is a non-negative self-adjoint operator on a metric measure space $X$ and $\psi \in C^{\infty}_c(\mathbb{R})$ with $\mathrm{supp}\,\psi \subset [-1/2,2]$. We believe that, by adapting the techniques as in the proof of Theorem \ref{MAIN3}, one can extend their result to a global Strichartz estimates for orthonormal systems of type (\ref{111})  associated with the flow $e^{- i t \phi(L)}$.
\end{rem}

\noindent {\bf The wave equation:} It is well known that the solution of the special Hermite wave equation\begin{align}
	i \partial_{tt} u(t, z)-{\mathcal{L}}u(t, z) &= 0,\quad z \in \mathbb{C}^n, \hspace{2pt} t \in \mathbb{R}\setminus \{0\}, \\
	\nonumber u(0,z) &= f(z)\\\nonumber\partial_tu(0,z)&=g(z).
\end{align}
 can be expressed as a superposition of waves generated by the propagators $e^{\pm i t \sqrt{\mathcal{L}}}.$ This corresponds to the case $\phi(r)=\sqrt{r}$ which satisfies condition (\ref{e1}) with $m=\frac{1}{2}$. We obtain the following Strichartz estimate by Theorem \ref{MAIN2}.
\begin{theorem}
		Let $n,p,q \geq1$ and $N\geq 2n$. Suppose $p, q $ satisfy  the conditions $1\leq q <\frac{n}{n-1},~ \frac{2}{p}+\frac{2n - 1}{q} = 2n-1,$ and let $\sigma=\frac{2n+1}{2}(1-\frac{1}{q})$. Then,
		\begin{align*}
			\left\|\sum_{j\in J} n_{j }\left|e^{-i t \sqrt{\mathcal{L}}}f_{j }\right|^{2}\right\|_{L^p ((-T_0,T_0), L^q(\mathbb{C}^n))} \leqslant C N^\sigma\left(\sum_{j \in J}\left|n_{j}\right|^{\beta}\right)^{1/\beta}
		\end{align*}
		holds for all orthonormal system $(f_j)_j$ in $L^2(\mathbb{C}^n)$ with $supp \hat{f}_j \subset \{(\mu,\nu): \sqrt{2|\nu|+n }\leq N\}$ and all sequence $ (n_{j })_j \subset \mathbb{C}$.
\end{theorem}

\noindent {\bf The Klein-Gordon equation:} The solution of the special Hermite Klein-Gordon equation\begin{align}
	i \partial_{tt} u(t, z)-{\mathcal{L}}u(t, z)- u(t, z) &= 0,\quad z \in \mathbb{C}^n, \hspace{2pt} t \in \mathbb{R}\setminus \{0\}, \\
	\nonumber u(0,z) &= f(z)\\\nonumber\partial_tu(0,z)&=g(z).
\end{align}
can be expressed as a superposition of waves generated by the propagators $e^{\pm i t \sqrt{1+\mathcal{L}}},$ which corresponds to the case $\phi(r)=\sqrt{r}$ satisfying (\ref{e1}) with $m=\frac{1}{2}$. We obtain the following Strichartz estimate by Theorem \ref{MAIN2}.

\begin{theorem}
		Let $n,p,q \geq1$ and $N\geq 2n$.  Suppose $p, q $ satisfy  the conditions $1\leq q <\frac{n}{n-1},~ \frac{2}{p}+\frac{2n - 1}{q} = 2n-1,$ and let $\sigma=\frac{2n+1}{2}(1-\frac{1}{q})$. Then,
		\begin{align*}
			\left\|\sum_{j\in J} n_{j }\left|e^{-i t \sqrt{1+\mathcal{L}}}f_{j }\right|^{2}\right\|_{L^p ((-T_0,T_0), L^q(\mathbb{C}^n))} \leqslant C N^\sigma\left(\sum_{j \in J}\left|n_{j}\right|^{\beta}\right)^{1/\beta}
		\end{align*}
		holds for all orthonormal system $(f_j)_j$ in $L^2(\mathbb{C}^n)$ with $supp \hat{f}_j \subset \{(\mu,\nu): \sqrt{2|\nu|+n }\leq N\}$, and all sequence $ (n_{j })_j \subset \mathbb{C}$.
\end{theorem}

\noindent {\bf The fractional Schr\"{o}dinger equation:} For $0\leq \varrho<1$, the solution to the special Hermite fractional Schr\"{o}dinger equation\begin{align}
	i \partial_{t} u(t, z)-{\mathcal{L}^\varrho}u(t, z) &= 0,\quad z \in \mathbb{C}^n, \hspace{2pt} t \in \mathbb{R}\setminus \{0\},\\
	\nonumber u(0,z) &= f(z) 
\end{align}
 is described by the unitary flow $e^{- i t\mathcal{L}^\varrho}$. This corresponds to the case 
 $\phi(r)={r}^\varrho$ satisfying (\ref{e1}) with $m=\varrho$. We obtain the following Strichartz estimate by Theorem \ref{MAIN2}.

\begin{theorem}
		Let $n,p,q \geq1$ and $N\geq 2n$.  Suppose $p, q $ satisfy  the conditions $1\leq q <\frac{n}{n-1},~ \frac{2}{p}+\frac{2n - 1}{q} = 2n-1,$ and let $\sigma=(2n(1-\varrho)+\varrho)(1-\frac{1}{q})$. Then 
		\begin{align}\label{OSinq2}
			\left\|\sum_{j\in J} n_{j }\left|e^{-i t \mathcal{L}^\varrho}f_{j }\right|^{2}\right\|_{L^p ((-T_0,T_0), L^q(\mathbb{C}^n))} \leqslant C N^\sigma\left(\sum_{j \in J}\left|n_{j}\right|^{\beta}\right)^{1/\beta}
		\end{align}
		holds for all orthonormal system $(f_j)_j$ in $L^2(\mathbb{C}^n)$ with $supp \hat{f}_j \subset \{(\mu,\nu): \sqrt{2|\nu|+n }\leq N\}$,  and all sequence $ (n_{j })_j \subset \mathbb{C}$ with $\beta = \frac{2q}{q+1}$.
\end{theorem}

\begin{rem}
Theorems \ref{main1}, \ref{MAIN2} and \ref{MAIN3} can be extended for a wider range of  $p,q \geq1$ such that  $\frac{2}{p}+\frac{2n - 1}{q} \geq 2n-1$,  using the inclusion properties of the spaces $L^p(-T_0,T_0)$ - spaces (see \cite{dl}). When $1 <q<\frac{n}{n-1}$, then there exists a $0<\beta \leq 1$ such that $\frac{2\beta}{p}+\frac{2n - 1}{q}=2n-1$. Since $\frac{p}{\beta} \geq p \geq 1$ the desired estimate follows from the fact that
$$\left\|\sum_{j} n_{j }\left|e^{-i t \phi(\mathcal{L})}f_{j }\right|^{2}\right\|_{L^p ((-T_0,T_0), L^q(\mathbb{C}^n))} \leq \left\|\sum_{j} n_{j }\left|e^{-i t \phi(\mathcal{L})}f_{j }\right|^{2}\right\|_{L^{\frac{p}{\beta}} ((-T_0,T_0), L^q(\mathbb{C}^n))}.$$
The case $p=1$ follows directly from the triangle inequality.
\end{rem}

\section{Endpoint case of Theorem \ref{main1}}\label{sec4}
	
By a standard duality argument in Schatten spaces (see \cite{frank} for  details), Theorem \ref{main1} can be equivalently restated in the following dual version.
\begin{theorem} Assume that $p', q', d \geq 1$ satisfy
	$$\frac{2n+1}{2} < p' \leq \infty \quad \mbox{and}\quad \frac{1}{q'} + \frac{n}{p'} = 1.$$
	we have
	\begin{align}\label{HIO}
		\left\|\int_{-\pi}^{\pi} e^{-it\mathcal{L}}V(t,z)e^{it\mathcal{L}} \,dt \right\|_{\mathcal{G}^{2q'}} \leq C \|V\|_{L^{p'} ((-\pi,\pi), L^{q'}(\mathbb{C}^n))}.
	\end{align}
	\end{theorem}
The estimate (\ref{Sch1}) at the endpoint $(q,\beta) = (\frac{2n+1}{2n -1},\frac{2q}{q+1})$ corresponds to the estimate (\ref{HIO}) at $2q'=2n+1$.  Hence, Theorem \ref{end} follows from the following result.

\begin{theorem} There exists  $0 \neq V \in  {L^{2n+1} ((-\pi,\pi), L^{\frac{2n+1}{2}}(\mathbb{C}^n))}$ such that
	\begin{align}
		Tr\left(\int_{-\pi}^{\pi} e^{it\mathcal{L}}V(t,z)e^{-it\mathcal{L}}\, dt \right)^{2n+1} = \infty.
	\end{align}
\end{theorem}
\begin{proof}
Define the operator
\begin{align}
	B_V := \int_{-\pi}^{\pi} e^{it\mathcal{L}}V(t,z)e^{-it\mathcal{L}}\, dt,
\end{align}
whose kernel can be calculated as 
$$B_V(z,w) = \int_{-\pi}^{\pi} (2 \pi i \sin t)^{-2n} \int_{\mathbb{C}^n} e^{-\frac{i}{2} \Im(z\cdot\bar{\zeta})}e^{-i \cot t \frac{|z-\zeta|^2}{4}} V(t,\zeta) e^{\frac{i}{2} \Im(w\cdot\bar{\zeta})}e^{  i \cot t \frac{|w-\zeta|^2}{4}} ~d\zeta~dt,$$
for a $V(t,z)$ such that the integral makes sense, we choose $V(t,z)$ precisely later.
For $f \in \mathcal{S}(\mathbb{C}^n)$ (the Schwartz space on $\mathbb{C}^n$), let $\mathcal{F}_s$ be the symplectic Fourier transform on $\mathbb{C}^n$ given by
\[
\mathcal{F}_s f(z)
= \int_{\mathbb{C}^n} 
f(w)\, e^{-\, \frac{i}{2}\, \Im(z\cdot \bar{w})} \, dw.
\]
Now the kernel of $B_V$ in the symplectic Fourier space is given by
\begin{align}\label{KUY}
	\widehat{B_V}(p,q)= &C\int_{-\pi}^{\pi} ( \cos t)^{-2n} \int_{\mathbb{C}^n} V(t,\zeta)  e^{i \tan t \frac{|p+\zeta|^2}{4}} e^{-i \tan t \frac{|q+\zeta|^2}{4}} e^{-\frac{i }{2}\Im((q-p)\cdot\bar{\zeta})} ~d\zeta~dt\nonumber\\
	=&C\int_{-\pi}^{\pi} (\cos t)^{-2n} e^{i \tan t (|p|^2 - |q|^2)} \int_{\mathbb{C}^n} V(t,\zeta)  e^{-\frac{i}{2} \Im((q-p)\cdot\bar{\eta})} ~d\zeta~dt
\end{align}
where $\eta = \eta_1 + i \eta_2$ is given by
\begin{align}\label{RTY}
	\eta_1 =\zeta_1 + \tan t\,\zeta_2  , \quad \eta_2 =- \tan t\,\zeta_1 + \, \zeta_2 .
\end{align}
Now, let $0 \neq V_1(t,z) \in L^\infty_c(\mathbb{R}\times\mathbb{C}^n)$ be a non-negative function such that $\widehat{V_1} = \mathcal{F}_t\mathcal{F}_{ s,z} V$ is non-negative.
Here $\mathcal{F}_{ s,z}$ is the symplectic Fourier transform with respect to $z$-variable, and $\mathcal{F}_t$ denotes the Fourier transform with respect to $t$-variable.
 Now consider 
$$V(t,z) = \chi_{(-\frac{\pi}{2}, \frac{\pi}{2})} V_1\left(\frac{\tan t}{4}, \zeta_1 + \tan t\,\zeta_2+ i(- \tan t\,\zeta_1 + \, \zeta_2)\right)\sec^2 t,$$
one can check that $V \in {L^{2n+1} ((-\pi,\pi), L^{\frac{2n+1}{2}}(\mathbb{C}^n))}$. Performing the change of variable (\ref{RTY}) in (\ref{KUY}), we obtain
\begin{align*}
	\widehat{B_V}(p,q)= &C \int_{-\frac{\pi}{2}}^{\frac{\pi}{2}} e^{i \frac{\tan t}{4} (|p|^2 - |q|^2)} \sec^2 t\int_{\mathbb{C}^n} V\left(\frac{\tan t}{4},\eta\right)  e^{-\frac{i}{2} \Im((q-p)\cdot\bar{\eta})} ~d\zeta~dt\\
	=&\widehat{V_1}(|q|^2-|p|^2, q-p),
\end{align*} where the last equality is obtained by a obvious change of variable.
Hence, we deduce that
$$Tr (B_V^{2n+1}) = \int_{\mathbb{C}^n} dp \int_{\mathbb{C}^n} dp_1 \cdots \int_{\mathbb{C}^n} dp_{2n}\,\, \widehat{B_V}(p,p_1) \widehat{B_V}(p_1,p_2)\cdots \widehat{B_V}(p_{2n},p).$$
Now, proceding as in the proof of Proposition 2 in \cite{frank}, we obtain $Tr (B_V^{2n+1}) = +\infty$.

\end{proof}

\section{Restriction theorem for the special Hermite spectral projections}\label{sec5}
In this section we obtain Schatten bounds the form (\ref{Q1}) for the special Hermite spectral projections. We follow the complex interpolation argument in Schatten spaces, due to Frank and Sabin \cite{frank1}. We begin by briefly recalling the notion of analytic families of operators in the sense of Stein, which plays a central role in this approach.

A family of operators $(G_\alpha)$ on $\mathbb{C}^n$ defined on a strip 
$a \le \mathrm{Re}\, \alpha \le b$ in the complex plane (with $a < b$) is said to be 
{analytic in the sense of Stein} if, for all simple functions 
$f, g$ on $\mathbb{C}^n$ (that is, functions that take a finite number of non-zero 
values on sets of finite measure in $\mathbb{C}^n$), the map
\[
\alpha \mapsto \langle g, G_\alpha f \rangle
\]
is analytic in the strip $a < \mathrm{Re}\, \alpha < b$, continuous on the closure 
$a \le \mathrm{Re}\, \alpha \le b$, and if
\[
\sup_{a \le \lambda \le b} 
\big| \langle g, G_{\lambda + i\tau} f \rangle \big| \le C(\tau),
\]
for some function $C(\tau)$ with at most a (double) exponential growth in $\tau$. We have the following Schatten space estimate obtained in \cite{frank1}.

\begin{prop}\cite{frank1}\label{F}
	Let $(G_\alpha)$ be an analytic family of operators on $\mathbb{C}^n$ in the sense of Stein, defined on the strip 
	$-\lambda_0 \le \mathrm{Re}\, \alpha \le 0$ for some $\lambda_0 > 1$. Assume that we have the bounds
	\[
	\|G_{i\tau}\|_{L^2 \to L^2} \le M_0 e^{a|\tau|}, 
	\qquad 
	\|G_{-\lambda_0 + i\tau}\|_{L^1 \to L^\infty} \le M_1 e^{b|\tau|}, 
	\qquad ~\mbox{for all}~\tau \in \mathbb{R},
	\]
	for some $a,b \ge 0$ and some $M_0, M_1 \ge 0$. Then, for all $W_1, W_2 \in L^{2\lambda_0}(\mathbb{C}^n)$, 
	the operator $W_1 G_{-1} W_2$ belongs to the Schatten class 
	$\mathcal{G}^{2\lambda_0}(L^2(\mathbb{C}^n))$ and we have the estimate
	\[
	\|W_1 G_{-1} W_2\|_{\mathcal{G}^{2\lambda_0}(L^2(\mathbb{C}^n))} 
	\le 
	M_0^{1 - \frac{1}{\lambda_0}} 
	M_1^{\frac{1}{\lambda_0}}
	\|W_1\|_{L^{2\lambda_0}(\mathbb{C}^n)} 
	\|W_2\|_{L^{2\lambda_0}(\mathbb{C}^n)}.
	\]
\end{prop}

We are now in a position to prove Theorem \ref{Thm1}. To this end, we  use above proposition by defining the same analytic family of operators introduced in  \cite{Ratna1}.
\vspace{6pt}

\noindent{\bf \emph{Proof of Theorem \ref{Thm1}} : } 
For $\alpha \in \mathbb{C}$ with $\Re\alpha > -1$, consider 
\begin{equation}
	\psi_k^{\alpha}(z) = \frac{\Gamma(k+1) \Gamma\alpha +1}{\Gamma(k+\alpha+1)} L_{k}^{\alpha}\left(\frac{1}{2}|z|^{2}\right) e^{-\frac{1}{4}|z|^{2}}.
\end{equation}
We then set
\begin{align*}
	G^\alpha_k f(z) =  f \times \psi_k^{\alpha + n}(z), \quad f\in \mathcal{S}(\mathbb{C}^n).
\end{align*}
The family $(G_k^\alpha)_\alpha$ forms an analytic family of operators in the strip $-n - \tfrac{1}{2} \leq \Re \alpha \leq 0$ (see \cite{Ratna1}). From \cite{Ratna1} and \cite{Than1}, we have the following estimates
\begin{align}\label{aw2}
	\|G_k^{-\lambda_0 + i \tau} f\|_{L^\infty(\mathbb{C}^n)} \leq C (1+ |\tau|)^{2/3} \|f\|_{L^1(\mathbb{C}^n)}, \quad 0 \leq \lambda_0 < n+1/3,
\end{align}
and 
\begin{align}\label{e2}
\hspace{-2cm}	\|G_k^{i \tau} f\|_{L^2(\mathbb{C}^n)} \leq C (1+ |\tau|)^{n} k^{-n} \|f\|_{L^2(\mathbb{C}^n)}.
\end{align}
Also we have $$ \mathcal{Q}_k f(z) = \frac{\Gamma(k+n)\Gamma(n)}{\Gamma(k+1)} G^{-1}_kf(z)$$
and $\frac{\Gamma(k+n)\Gamma(n)}{\Gamma(k+1)} \leq C k^{n-1}.$ Thus by Proposition \ref{F}, we obtain
\begin{align}\label{P1}
	\left\|W_1 \mathcal{Q}_k {W_2}\right\|_{\mathcal{G}^{2\lambda_0}\left(L^{2}\left( \mathbb{C}^{n} \right)\right)} \leq Ck^{\frac{n}{\lambda_0} -1} \| W_1\|_{L^{2\lambda_0}(\mathbb{C}^n)} \|W_2\|_{L^{2\lambda_0}(\mathbb{C}^n)},
\end{align}
for $1 \leq \lambda_0 < n+\frac{1}{3}$.

To complete the proof of the theorem, we observe that $W_1\mathcal{Q}_k W_2 = (W_1 \mathcal{Q}_k)(\overline{W_2} \mathbb{Q}_k)^*$. The operator $W\mathcal{Q}_k$ acts from $L^2(\mathbb{C}^n)$ to  $L^2(\mathbb{C}^n)$ as an integral operator with integral kernel $K(w,z) = W(z)\varphi_k(z-w) e^{i Im(z\bar{w})}$, where $z, w \in \mathbb{C}^n$. If $W \in L^2(\mathbb{C}^n)$, then $W\mathcal{Q}_k$ is Hilbert-Schmidt, in fact,
\begin{align*}
\|W \mathcal{Q}_k\|_{\mathcal{G}^2(L^2(\mathbb{C}^n))}^2 
= \int_{\mathbb{C}^n} |W(z)|^2 
\left( \int_{\mathbb{C}^n} |\varphi_k(z - w)|^2 \, dw \right) dz
 \leq C k^{n-1} \|W\|^2_{L^2(\mathbb{C}^n)},
\end{align*}
since $\|\varphi_k\|_{L^2(\mathbb{C}^n)} \leq C k^{\frac{n-1}{2}}$. Now, by H\"older's inequality for trace ideals, we get
\begin{align}\label{P2}
	\|W_1 \mathcal{Q}_k W_2\|_{\mathcal{G}^1(L^2(\mathbb{C}^n))}
	\leq C k^{n-1} \|W_1\|_{L^2(\mathbb{C}^n)}\|W_2\|_{L^2(\mathbb{C}^n)}.
\end{align}
By complex interpolation between (\ref{P2}) and (\ref{P1}) with $\lambda_0$ close to $ n + \frac{1}{3}$, we get 
	\begin{align*}
	\left\|W_1 \mathcal{Q}_k {W_2}\right\|_{\mathcal{G}^{\frac{(6n - 4)q}{6n-1-3q}}\left(L^{2}\left( \mathbb{C}^{n} \right)\right)} \leq Ck^{\frac{n}{q} - 1} \| W_1\|_{L^{2q}(\mathbb{C}^n)} \|W_2\|_{L^{2q}(\mathbb{C}^n)}.
\end{align*} for $1 \leq q <  \frac{3n+1}{3}$.  By setting $2q = \frac{2p}{2-p}$, we obtain the estimate (\ref{Q1}).
$\hfill{\square}$	
	
\vspace{6pt}

In the proof of Theorem \ref{Thm1}, we do not expect a global dispersive estimate of the form (\ref{aw2}) when $\alpha = -n - \tfrac{1}{2} + i\tau, \tau\in \mathbb{R}$ (see Lemma 1 of \cite{Markett}). However, the estimate behaves well on compact subsets (see \cite{Than3}). Using this fact we derive an improved local estimate as in Theorem \ref{Thm3}.

\vspace{6pt}
\noindent{\bf \emph{Proof of Theorem \ref{Thm3}} : }  Let $B$ be a fixed compact set on $\mathbb{C}^n$. We consider the following analytic family of operators, for  $\alpha \in \mathbb{C}$, let
\begin{align*}
	G^\alpha_k f(z) = (\chi_B f )\times \psi_k^{\alpha + n}(z),\quad f\in \mathcal{S}(\mathbb{C}^n).
\end{align*}
We have the following estimates
\begin{align}
	\|G_k^{-(n+\frac{1}{2}) + i \tau} f\|_{L^\infty(\mathbb{C}^n)} \leq& C_B (1+ |\tau|)^{1/2} \|f\|_{L^1(\mathbb{C}^n)},\label{E1} \\
	\hspace{.2cm}	\|G_k^{i \tau} f\|_{L^2(\mathbb{C}^n)} \leq &C_B (1+ |\tau|)^{n} k^{-n} \|f\|_{L^2(\mathbb{C}^n)}\label{E2}.
\end{align}
Estimate (\ref{E1}) can be derived from estimate (3.7) in \cite{Than3}, while estimate (\ref{E2}) follows from the better estimate (\ref{e2}). For $\alpha = -1$, we have $$ \mathcal{Q}_k \chi_Bf(z) = \frac{\Gamma(k+n)\Gamma(n)}{\Gamma(k+1)} G^{-1}_kf(z).$$
 Now, by Proposition \ref{F} we obtain
\begin{align}\label{EA2}
	\left\|W_1 \chi_B\mathcal{Q}_k\chi_B {W_2}\right\|_{\mathcal{G}^{2n+1}\left(L^{2}\left( \mathbb{C}^{n} \right)\right)} \leq Ck^{-\frac{1}{2n+1}} \| W_1\|_{L^{2n+1}(\mathbb{C}^n)} \|W_2\|_{L^{2n+1}(\mathbb{C}^n)}.
\end{align}

We observe that $W_1\chi_B\mathcal{Q}_k\chi_B W_2 = (W_1 \chi_B\mathcal{Q}_k)(\overline{W_2}\chi_B \mathbb{Q}_k)^*$. The operator $W\chi_B\mathcal{Q}_k$ acts on $L^2(\mathbb{C}^n)$  as an integral operator with  kernel $W(z)\chi_B(z)\varphi_k(z-w) e^{i Im(z\bar{w})}$, where $z, w \in \mathbb{C}^n$. If $W \in L^2(\mathbb{C}^n)$, then $W\chi_B\mathcal{Q}_k$ is Hilbert-Schmidt, with
\begin{align*}
	\|W \chi_B \mathcal{Q}_k\|_{\mathcal{G}^2(L^2(\mathbb{C}^n))}^2 
	\leq C_B k^{n-1} \|W\|^2_{L^2(\mathbb{C}^n)}.
\end{align*}
 Now, by H\"older's inequality for trace ideals, it follows that
\begin{align}\label{EA1}
	\|W_1 \chi_B\mathcal{Q}_k \chi_BW_2\|_{\mathcal{G}^1(L^2(\mathbb{C}^n))}
	\leq C_B k^{n-1} \|W_1\|_{L^2(\mathbb{C}^n)}\|W_2\|_{L^2(\mathbb{C}^n)}.
\end{align}

Furthermore, it is evident that $\mathcal{Q}_k\chi_B$ is bounded on $L^2(\mathbb{C}^n)$; that is, $\|\mathcal{Q}_k\chi_Bf\|_{L^2(\mathbb{C}^n)} \leq C_B \|f\|_{L^2(\mathbb{C}^n)}$. If $W_1, W_2 \in L^\infty(\mathbb{C}^n)$, then applying H\"older's inequality, we obtain
\begin{align}\label{EA3}
	\|W_1 \chi_B\mathcal{Q}_k \chi_BW_2\|_{\mathcal{G}^\infty(L^2(\mathbb{C}^n))}
	\leq C_B  \|W_1\|_{L^\infty(\mathbb{C}^n)}\|W_2\|_{L^\infty(\mathbb{C}^n)}.
\end{align}
By complex interpolation between (\ref{EA2}) and (\ref{EA1}) we get the estimate (\ref{Q2}) for $1 \leq p \leq \frac{2(2n+1)}{2n+3}$, and interpolating (\ref{EA2}) and (\ref{EA3}) we obtain (\ref{Q2}) for $\frac{2(2n+1)}{2n+3} \leq p \leq 2$, which completes the proof of the theorem.
$\hfill{ \square}$	

\section*{Acknowledgments}

The first author thanks the  Indian Institute of Technology Guwahati for the support provided during the period of this work.

\end{document}